\documentclass[12pt,oneside,reqno]{amsart}

\usepackage[T1]{fontenc}
\usepackage{lmodern}

\usepackage{amssymb,amsfonts,amsmath,amsthm,mathtools}

\usepackage{geometry}

\usepackage{xcolor}

\usepackage{tabularray}

\usepackage{enumitem}

\usepackage{hyperref}

\hypersetup{%
  colorlinks=true,
  linkcolor=blue,
  citecolor=blue,
  urlcolor=blue,
  pdfauthor={Saeid Jafari and Ernesto Lax},
  pdftitle={A note on the class of sober rings}
}

\def\qed{\ifhmode\textqed\fi
	\ifmmode\ifinner\hfill\quad\qedsymbol\else\dispqed\fi\fi}
\def\textqed{\unskip\nobreak\penalty50
	\hskip2em\hbox{}\nobreak\hfill\qedsymbol
	\parfillskip=0pt \finalhyphendemerits=0}
\def\dispqed{\rlap{\qquad\qedsymbol}}

\setlist[enumerate,1]{%
	label={\normalfont\arabic*.},
	left={\parindent},
	itemsep={15pt}%
}%
\setlist[enumerate,2]{%
	label={\normalfont(\roman*)},
	left={\parindent},
	itemsep={15pt}%
}%

\newcommand{\MBB}{\mathbb}
	
	\def\ZZ{\MBB Z}

\newcommand{\MFK}{\mathfrak}
  \def\m{\MFK m}
  \def\p{\MFK p}
  \def\q{\MFK q}

\let\sect=\cap

\let\Sect=\bigcap

\let\epsilon=\varepsilon
\let\phi=\varphi
\let\kappa=\varkappa

\let\To=\longrightarrow

\let\Implies=\Longrightarrow

\let\iff=\Longleftrightarrow

\let\:=\colon

\DeclareMathOperator{\Spec}{Spec}
\DeclareMathOperator{\Max}{Max}

\theoremstyle{plain}
	\newtheorem{Theorem}{Theorem}[section]
	\newtheorem{Lemma}[Theorem]{Lemma}
	\newtheorem{Proposition}[Theorem]{Proposition}
	\newtheorem{Corollary}[Theorem]{Corollary}

\theoremstyle{definition}
	\newtheorem{Definition}[Theorem]{Definition}
	\newtheorem{Example}[Theorem]{Example}
	
	\newtheorem{Remark}[Theorem]{Remark}
\begin{document}
\title{A note on the class of sober rings}

\author{Saeid Jafari}
\address{Saeid Jafari, Mathematical and Physical Science Foundation, Sidevej 5, 4200 Slagelse, Denmark.
Also affiliated with: Arul Anandar College (Autonomous),
Karumathur--625514, Madurai Dt., Tamil Nadu, India.}
\email{jafaripersia@gmail.com, saeidjafari@topositus.com}

\author{Ernesto Lax}
\address{Ernesto Lax, Department of Mathematics and Computer Sciences, Physics and Earth Sciences, University of Messina, Viale Ferdinando Stagno d'Alcontres 31, 98166 Messina, Italy}
\email{erlax@unime.it}


\begin{abstract}
  We study sober rings, a class of commutative rings in which every prime non-maximal ideal is not the intersection of the prime ideals strictly containing it. We establish basic structural properties of sober rings and investigate connections to other known classes of rings. We show that sobriety is preserved under quotients and localizations and we examine the behavior of sobriety for polynomial extensions.
\end{abstract}

\maketitle

\section{Introduction}
The study of prime ideals and their intersections is a central theme in commutative algebra. A particularly relevant notion in this context is that of a \emph{Goldman ideal}, or \emph{$G$-ideal}: a prime ideal of a ring $R$ that is the contraction of a maximal ideal of the polynomial ring $R[x]$. These ideals have been investigated from both algebraic and topological perspectives (see \cite{FL16,Fontana,FM76,P1,P2}). A well-known characterization (see \cite{Kap}) states that a prime ideal $P\in\Spec(R)$ is a $G$-ideal if and only if it is not the intersection of the prime ideals properly containing it. This characterization suggests a natural ring-theoretic property defined in terms of the structure of the prime spectrum. 

In this paper, we study \emph{sober rings}, namely rings in which every prime non-maximal ideal is a $G$-ideal, that is
\[
P\neq \Sect_{\substack{Q\in\Spec(R)\\P\subsetneq Q}}Q
\]
for every prime non-maximal ideal $P$ of $R$. This algebraic notion is inspired by the concept of sobriety in topology, where a topological space is sober if every non-empty irreducible closed subset is the closure of a unique point. Here, however, sobriety is expressed directly through the inclusion relations among prime ideals and their intersections. The aim of this paper is to investigate the behavior of sober rings under standard ring constructions and to identify conditions that guarantee sobriety.

In Section \ref{sec:Basics} we establish basic properties of the class of sober rings, including its behavior under finite direct products, and examine the role played by Krull dimension and the number of maximal ideals. In particular, we prove that every semilocal ring of Krull dimension $1$ is sober (Proposition \ref{prop:Semilocal-sober}). We then characterize sobriety for Dedekind domains that are not fields: such a domain is sober if and only if it is semilocal, or equivalently, if its Jacobson radical is non-zero (Proposition \ref{prop:Dedekind}). We also show that a Jacobson ring of positive dimension cannot be sober (\ref{prop:JAC-not_sober}).

In Section \ref{sec:Quotients} and Section \ref{sec:Localization} we investigate the behavior of sobriety under quotients and localizations respectively. We prove that every quotient of a sober ring is sober (Theorem \ref{cor:quotient-sober}) and that sobriety is preserved under localization at an arbitrary multiplicatively closed subset not containing zero \ref{thm:localization-sobriety}. These results show that sobriety is stable under two fundamental constructions in commutative algebra.

Finally, in Section \ref{sec:PolynomialRings} we consider polynomial extensions. Sobriety is not preserved by adjoining an indeterminate in general: for instance, if $K$ is a field, then $K$ is sober whereas $K[x]$ is not. We therefore examine conditions ensuring that $R[x]$ is sober. We formulate equivalent conditions through separating polynomial elements (Proposition \ref{prop:polynomial-sobriety-criterion}) and intersections of non-zero prime ideals in suitable quotient domains (Proposition \ref{prop:poly-quotient-intersection}). These criteria describe the additional separation property that must hold in the prime spectrum of the polynomial ring. Moreover, we give a sufficient condition in terms of finite \emph{cofinal families of prime ideals} (Theorem \ref{thm:finite-cofinal-sobriety}).

Throughout the paper, all rings are assumed to be commutative rings with identity $1\neq 0$, unless otherwise stated, and all ring homomorphisms preserve identity. We denote by $\Spec(R)$ the prime spectrum of a ring $R$ and by $\Max(R)$ its set of maximal ideals.

\section{Definition and initial results}\label{sec:Basics}
\begin{Definition}\label{def:sobriety}
  A ring $R$ is called a \emph{sober ring} if every non-maximal prime ideal of $R$ is a $G$-ideal, that is, if for every $P\in\Spec(R)\setminus\Max(R)$ one has
  \[
  P\neq \Sect_{\substack{Q\in\Spec(R) \\ P\subsetneq Q}}Q.
  \]
  Equivalently, $R$ is sober if for every prime non-maximal ideal $P$, there exists an element $a\in R$ such that
  \[
  a\notin P \quad\text{and}\quad a\in Q,
  \]
  for every prime ideal $Q$ strictly containing $P$. Such an element will be called a \emph{sober witness} for the ideal $P$. 
\end{Definition}

\begin{Remark}
  The intersection in Definition \ref{def:sobriety} is well-defined. Indeed, if $P$ is a non-maximal prime ideal, then there exists a maximal ideal $\mathfrak m$ such that $P\subsetneq\mathfrak m$, and hence the family of prime ideals properly containing $P$ is non-empty.
\end{Remark}

\begin{Proposition}\label{prop:radical}
  Let $R$ be a ring and let $P \in \Spec(R)$ be non-maximal. Then
  \[
  \Sect_{\substack{Q\in\Spec(R)\\P\subsetneq Q}}Q
  =
  \sqrt{\,\Sect_{\substack{Q\in\Spec(R)\\P\subsetneq Q}}Q\,}.
  \]
  In particular, the intersection appearing in Definition \ref{def:sobriety} is a radical ideal containing $P$.
\end{Proposition}
\begin{proof}
  An intersection of prime ideals is always radical. Since every prime ideal $Q$ in the intersection contains $P$, their intersection automatically contains $P$.
\end{proof}

First recall that the {\em Krull dimension} of a ring $R$, denoted by $\dim(R)$, is the supremum of the lengths of the finite ascending chains
\[
P_0\subset P_1 \subset \ldots \subset P_{k-1} \subset P_k,
\]
of prime ideals of $R$. A ring is called {\em zero-dimensional} if it has Krull dimension zero. In a zero-dimensional ring every prime ideal is also maximal, hence the next result immediately follows.

\begin{Proposition}\label{prop:0-dim_sober}
	Every zero-dimensional ring is sober.
\end{Proposition}

\begin{Corollary}\label{cor:Artin-comm_sober}
	Let $R$ be an Artinian ring. Then $R$ is a sober ring.
\end{Corollary}

\subsection{Finite direct products}
The prime spectrum of a finite product is the disjoint union of the spectra of its factors. The next result makes the corresponding algebraic consequence explicit.

\begin{Proposition}\label{prop:direct-product}
  Let $R$ and $S$ be non-zero commutative rings with identity. Then $R\times S$ is sober if and only if both $R$ and $S$ are sober.
\end{Proposition}
\begin{proof}
  Every prime ideal of $R\times S$ has exactly one of the forms
  $P\times S$, where $P\in\Spec(R)$, or $R\times Q$, where
  $Q\in\Spec(S)$. Moreover, $P\times S$ is maximal if and only if
  $P$ is maximal, and similarly for $R\times Q$.
  
  Fix a non-maximal $P\in\Spec(R)$. The prime ideals of $R\times S$
  strictly containing $P\times S$ are exactly the ideals
  $Q\times S$ with $Q\in\Spec(R)$ and $P\subsetneq Q$. Therefore
  \[
  \Sect_{\substack{\q\in\Spec(R\times S)\\P\times S\subsetneq\q}}\q
  =
  \left(\Sect_{\substack{Q\in\Spec(R)\\P\subsetneq Q}}Q\right) \times S.
  \]
  Thus the sobriety condition for $P\times S$ is equivalent to the
  sobriety condition for $P$. The same argument applies to primes of
  the form $R\times Q$. Hence $R\times S$ is sober exactly when both
  factors are sober.
\end{proof}

\begin{Corollary}
  For a finite family of non-zero rings $R_1,\ldots,R_n$,
  \[
  \prod_{i=1}^n R_i\text{ is sober} \quad\iff\quad R_i\text{ is sober for every }i.
  \]
\end{Corollary}

\subsection{Rings of dimension at most 1}
An Artinian ring has only finitely many maximal ideals. More in general, a ring with such a property is called a {\em semilocal ring}. It is natural to ask, whenever the semilocality of the ring guarantees sobriety. The next result gives a partial answer to this problem.

\begin{Proposition}\label{prop:Semilocal-sober}
	Let $R$ be a semilocal ring with $\dim(R)=1$. Then $R$ is a sober ring.
\end{Proposition}
\begin{proof}
Let $P\in\Spec{(R)}\setminus\Max{(R)}$. Since $P$ is not maximal, there exists a maximal ideal $\m$ such that $P\subsetneq \m$.

Consider the quotient ring $R/P$, which is a domain since $P$ is prime. Observe that $R/P$ is a semilocal ring, since its maximal ideals are in correspondence with the maximal ideals of $R$ containing $P$. Moreover, $\dim(R/P)\geq 1$ since $P$ is not maximal. On the other hand, one has $\dim(R/P)\leq \dim(R)=1$, hence
\[
\dim(R/P)=1.
\]

Let $\m_1/P,\ldots,\m_r/P$ be the maximal ideals of $R/P$, where $\m_i$ is a maximal ideal of $R$ containing $P$. For all $1\leq i\leq r$, since $\m_i/P$ is non-zero, we can choose $\overline{x}_i\in \m_i/P$ such that $\overline{x}_i\neq \overline{0}$. Since $R/P$ is a domain, it follows that
\[
\overline{x}=\overline{x}_1\cdots\overline{x}_r\neq \overline{0}.
\]
Moreover, $\overline{x}\in\Sect_{i=1}^r\m_i/P$, hence
\[
\Sect_{i=1}^r\m_i/P\neq(\overline{0}).
\]
By the correspondence between ideals of $R/P$ and ideals of $R$ containing $P$, we have
\[
P \neq \Sect_{i=1}^r \m_i.
\]
Since $\dim R=1$, every prime ideal properly containing $P$ must be maximal. Hence
\[
\Sect_{i=1}^r\m_i = \Sect_{\substack{Q \in \Spec(R) \\ P \subsetneq Q}} Q.
\]
It follows that $P$ differs from the intersection of the prime ideals of $R$ properly containing it, and therefore $R$ is sober.
\end{proof}

\begin{Remark}
	Observe that the semilocality of the ring in Proposition \ref{prop:Semilocal-sober} is fundamental. In fact, a $1$-dimensional ring with an infinite number of maximal ideals may not be sober in general. 
\end{Remark}\pagebreak

\begin{Example}\label{ex:ZZ-not_sober}
	The ring $\ZZ$ of integers is not a sober ring. In fact, in $\ZZ$ the non-zero prime ideals are maximal and of the form $(p)$, with $p$ a prime number. Hence, the only prime non-maximal ideal is $(0)$. Clearly
	\[
	(0) = \Sect_{\substack{p\in\ZZ\\ \text{$p$ prime}}} (p),
	\]
	hence $\ZZ$ is not sober.
\end{Example}

Recall that a Dedekind domain is an integral domain which is either a field or such that every non-zero proper ideal factors into the product of prime ideals. For a Dedekind domain $R$ one has $\dim R\leq 1$.

The next result states a more general property then the one illustrated in Example \ref{ex:ZZ-not_sober}. Recall that for a unitary ring $R$, the Jacobson radical $J(R)$ of $R$ is the intersection of all maximal ideals of $R$. 

\begin{Proposition}\label{prop:Dedekind}
	Let $R$ be a Dedekind domain which is not a field. Then the following conditions are equivalent: 
  \begin{enumerate}
    \item $R$ is sober;
    \item $J(R)\neq (0)$;
    \item $R$ is semilocal.
  \end{enumerate}
\end{Proposition}

In order to prove this fact we need the following lemma.

\begin{Lemma}\label{lemma:Dedekind-Jnon-zero<->MaxInfinite}
  Let $R$ be a Dedekind domain. Then $R$ has infinitely many maximal ideals if and only if $J(R)=(0)$.
\end{Lemma}
\begin{proof}
  Suppose $R$ has an infinite number of maximal ideals, and assume by contradiction that $J(R)\neq (0)$. Since the Jacobson radical is not the zero ideal, there is at least an element $x\in J(R)$ such that $x\neq 0$, thus $x\in\m$ for all $\m\in\Max(R)$, in particular $R$ is not a field. On the other hand consider the principal ideal $(x)$, since $R$ is a Dedekind domain, the principal ideal $(x)\neq(0)$ can be factored as
  \[
  (x)=\p_1 \cdots \p_n,
  \]
  where $\p_1,\ldots,\p_n$ are distinct non-zero prime ideals. Let $\m$ be a maximal ideal containing $x$, then $(x) \subseteq \m$ and since $\m$ is also prime, it follows that
  \[
  \p_i \subseteq \m
  \]
  for some $i$. As every non-zero prime ideal of a Dedekind domain is
  maximal, we obtain $\m = \p_i$. Thus $x$ belongs to only finitely many maximal ideals, namely among $\p_1,\ldots,\p_n$, but this contradicts the fact that $x\in\m$ for every $\m\in\Max(R)$.
  
  Conversely, suppose $J(R)\neq (0)$ and assume by contradiction that $R$ has only finitely many distinct maximal ideals $\m_1,\ldots,\m_n$. Hence, consider the product
  \[
  I=\m_1 \cdots \m_n.
  \]
  It is clear that $I\neq (0)$, but since $I\subseteq \m_i$ for all $i$, one has
  \[
  (0)\neq I \subseteq J(R) = (0),
  \]
  a contradiction.
\end{proof}

\begin{proof}[Proof of Proposition \ref{prop:Dedekind}] 
  The implications {2. {$\Implies$} 3.} and {3. {$\Implies$} 1.} follow from Lemma \ref{lemma:Dedekind-Jnon-zero<->MaxInfinite} and Proposition \ref{prop:Semilocal-sober} respectively. Thus, it suffices to show that {1. {$\Implies$} 2.}
  
  Suppose $R$ is sober and prove $J(R)\neq(0)$. Since $R$ is a Dedekind domain which is not a field, one has $\dim R=1$. This implies that every non-zero prime ideal is also maximal. Thus, by sobriety
  \[
  (0) \neq \Sect_{\substack{P\in\Spec(R) \\ (0)\subsetneq P}} P,
  \]
  where the intersection is actually taken over every maximal ideal of $R$, hence
  \[
  (0) \neq \Sect_{\m\in\Max(R)} \m = J(R),
  \]
  and this proves the implication 1. $\Implies$ 2.
\end{proof}

\subsection{A glimpse to rings of greater dimension}
We conclude this section by considering rings of dimension at least $1$. In particular, we focus on Jacobson rings. Recall that a commutative ring $R$ is called a {\em Jacobson ring} if every prime ideal of $R$ is the intersection of the maximal ideals containing it. The following proposition shows that this condition is incompatible with sobriety.

\begin{Proposition}\label{prop:JAC-not_sober}
	Let $R$ be a Jacobson ring, which is not a field. Then $R$ is not a sober ring.
\end{Proposition}
\begin{proof}
	Observe that, since $\dim(R)\geq 1$, one has $\Spec(R)\setminus\Max(R)\neq\emptyset$.\linebreak
	Let $P\in\Spec(R)\setminus\Max(R)$. Since $R$ is a Jacobson ring, $P$ is the intersection of the maximal ideals properly containing it. Hence, let $\mathcal{M} = \{M\in\Max(R)\ :\ M\supset P\}$, we have
	\[
	P=\Sect \mathcal{M}.
	\]
	Let now $\mathcal{P}=\{Q\in\Spec(R)\ :\ Q\supset P\}$, the family of prime ideals properly containing $P$. It is clear that $P\subseteq \Sect\mathcal{P}$. Moreover, since $\mathcal{M}\subset\mathcal{P}$, one has $\displaystyle{\Sect\mathcal{P} \subseteq \Sect\mathcal{M}}$.\linebreak
	Hence,
	\[
	P \subseteq {\Sect\mathcal{P}} \subseteq {\Sect\mathcal{M}} = P.
	\]
	Therefore, $P=\displaystyle\Sect\mathcal{P}$ and $R$ is not sober.
\end{proof}

\section{Quotients of sober rings}\label{sec:Quotients}
In this section, we study the behavior of sobriety under quotients. Let $R$ be a ring and let $I$ be an ideal of $R$. Using the inclusion-preserving correspondence between the prime ideals of $R/I$ and the prime ideals of $R$ containing $I$, we establish a characterization for the sobriety of $R/I$. As a consequence, we show that sobriety is preserved under arbitrary quotients.

\begin{Theorem}\label{thm:quotient-sobriety-criterion}
  Let $R$ be a ring and let $I$ be an ideal of $R$. Then the following statements are equivalent:
  \begin{enumerate}
  \item $R/I$ is sober;
  \item for every $P\in\Spec(R)$ with $I\subseteq P$ such that $P/I$ is not maximal in $R/I$,
  \[
  P\neq \Sect_{\substack{Q\in\Spec(R) \\ P\subsetneq Q}} Q;
  \]
  \item for every $P\in\Spec(R)$ with
  $I\subseteq P$ and $P\notin\Max(R)$,
  \[
  P \neq \Sect_{\substack{Q\in\Spec(R) \\ P\subsetneq Q}} Q.
  \]
  \end{enumerate}
\end{Theorem}
\begin{proof}
  We use the standard correspondence theorem for prime ideals of quotient rings. Consider the map
  \begin{align*}
    \Phi:\Spec(R/I) & \To {P\in\Spec(R)\subseteq P},\\
    \p & \longmapsto\pi^{-1}(\p),
  \end{align*}
  where
  \[
  \pi: R \To R/I
  \]
  is the canonical quotient projection map. By the correspondence theorem for prime ideals, the map $\Phi$ is an inclusion-preserving bijection, with inverse given by
  \[
  P \To P/I.
  \]
  In particular, for prime ideals $P,Q$ containing $I$,
  \[
  P/I\subsetneq Q/I \quad \iff \quad P\subsetneq Q.
  \]
  Moreover,
  \[
  P/I\in\Max(R/I) \quad\iff\quad P\in\Max(R).
  \]\smallskip

  Thus, let $P\in\Spec(R)$ satisfy $I\subseteq P$ and suppose $P/I$ is not maximal in $R/I$. The prime ideals of $R/I$ properly containing $P/I$ are precisely the ideals $Q/I$, where
  \[
  Q\in\Spec(R) \qquad\text{and}\qquad  P\subsetneq Q.
  \]
  Since $I\subseteq P$, every such $Q$ automatically contains $I$. Therefore,
  \begin{align*}
  \Sect_{\substack{\q\in\Spec(R/I) \\ P/I\subsetneq\q}} \q 
  & = \Sect_{\substack{Q\in\Spec(R) \\ P\subsetneq Q}} Q/I \\
  & = \Bigg(\Sect_{\substack{Q\in\Spec(R) \\ P\subsetneq Q}} Q \Bigg) \big/ I
  \end{align*}
  Hence
  \[
  P/I \neq \Sect_{\substack{\q\in\Spec(R/I) \\ P/I\subsetneq\q}} \q
  \]
  if and only if
  \[
  P\neq \Sect_{\substack{Q\in\Spec(R) \\ P\subsetneq Q}} Q.
  \]
  Thus, $R/I$ is sober if and only if the stated inequality
  holds for every prime ideal $P\supseteq I$ for which $P/I$
  is non-maximal. This proves the equivalence 1. $\iff$ 2.
  
  Finally, $P/I$ is maximal in $R/I$ if and only if $P$ is
  maximal in $R$. Hence condition 2. is equivalent to condition
  3.
\end{proof}

The previous result yields immediately the following preservation property.

\begin{Corollary}\label{cor:quotient-sober}
  Every quotient ring of a sober ring is sober.
\end{Corollary}
\begin{proof}
  Let $R$ be sober and let $I$ be an ideal of $R$. By Theorem~\ref{thm:quotient-sobriety-criterion}, it suffices to consider a prime
  ideal $P\in\Spec(R)$ such that
  \[
  I\subseteq P \quad\text{and}\qquad P\notin\Max(R).
  \]
  Since $R$ is sober,
  \[
  P\neq
  \Sect_{\substack{Q\in\Spec(R) \\ P\subsetneq Q}} Q.
  \]
  Theorem~\ref{thm:quotient-sobriety-criterion} therefore implies that $R/I$ is
  sober.
\end{proof}

\begin{Remark}
  The preceding corollary shows that no additional restriction on the ideal $I$ is needed when the original ring $R$ is sober. Thus, every quotient of a sober ring is sober, including the zero ring when $I=R$.
  
  We emphasize that this preservation property is specific to the defining intersection condition for sober rings. It should not be confused with the behavior of other ring-theoretic properties, such as being Noetherian, regular, or integrally closed, which are not, in general, preserved under arbitrary quotients without additional hypotheses.
\end{Remark}

\section{Localization of sober rings}\label{sec:Localization}
Let $R$ be a ring and let $S$ be a multiplicatively closed subset of $R$ such that $1\in S$ and $0\notin S$. In this section we investigate conditions under which the sobriety of $R$ is preserved by the localization $S^{-1}R$.

Recall that the prime ideals of $S^{-1}R$ are precisely the ideals $S^{-1}P$, where $P\in\Spec(R)$ and $P\sect S=\emptyset$ (see \cite{AtiyahMacdonald}). Moreover, the correspondence
\[
P \longmapsto S^{-1}P
\]
preserves strict inclusions among prime ideals disjoint from $S$.

The following theorem shows that no additional hypothesis on
the multiplicative set is required.

\begin{Theorem}\label{thm:localization-sobriety}
  Let $R$ be a sober ring and let $S\subseteq R$ be a multiplicatively closed set satisfying $1\in S$ and $0\notin S$. Then the localization $S^{-1}R$ is sober.
\end{Theorem}
\begin{proof}
  Let $\p$ be a non-maximal prime ideal of $S^{-1}R$. By the prime ideal correspondence for localization, there exists a prime ideal $P$ of $R$ such that $P\sect S=\emptyset$ and $\p=S^{-1}P$.
  Since $\p$ is non-maximal, there exists a prime ideal $\q$ of $S^{-1}R$ such that
  $\p\subsetneq\q$. Again, by the prime correspondence, one has $\q=S^{-1}Q$ for a $Q\in\Spec(R)$ such that $Q\sect S=\emptyset$. We obtain
  \[
  P\subsetneq Q.
  \]
  In particular, $P$ is not maximal in $R$. Since $R$ is sober, we have
  \[
  P\neq \Sect_{\substack{Q\in\Spec(R)\\P\subsetneq Q}} Q.
  \]
  Every prime ideal in this intersection contains $P$. Consequently, there exists an element
  \[
  a\in \Sect_{\substack{Q\in\Spec(R)\\P\subsetneq Q}}Q \setminus P.
  \]
  Thus $a\in Q$ for every prime ideal $Q$ such that $P\subsetneq Q$, while $a\notin P$.
  
  Consider the element $a/1\in S^{-1}R$. Since $a\notin P$ and $P\sect S=\emptyset$, the localization criterion for ideal membership gives
  \[
  \frac{a}{1}\notin S^{-1}P.
  \]
  On the other hand, every prime ideal of $S^{-1}R$ strictly containing $S^{-1}P$ has the form $S^{-1}Q$, where $Q$ is a prime ideal of $R$ satisfying
  \[
  P\subsetneq Q
  \qquad\text{and}\qquad
  Q\sect S=\emptyset.
  \]
  For every such $Q$, we have $a\in Q$, and hence
  \[
  \frac{a}{1}\in S^{-1}Q.
  \]
  It follows that
  \[
  \frac{a}{1}\in
  \Sect_{\substack{\q\in\Spec(S^{-1}R)\\
  S^{-1}P\subsetneq\q}}\q,
  \]
  but
  \[
  \frac{a}{1}\notin S^{-1}P.
  \]
  Therefore,
  \[
  S^{-1}P\neq
  \Sect_{\substack{\q\in\Spec(S^{-1}R)\\
  S^{-1}P\subsetneq\q}}\q.
  \]
  
  Since $\p$ was an arbitrary non-maximal prime ideal
  of $S^{-1}R$, the localization $S^{-1}R$ is sober.
\end{proof}


\begin{Remark}
  The proof of Theorem \ref{thm:localization-sobriety} does not require the localization map
  \[
  R \To S^{-1}R
  \]
  to be injective. It uses only the prime ideal correspondence for localization and the fact that, for a prime ideal $P$ disjoint from $S$,
  \[
  \frac{a}{1}\in S^{-1}P  \quad\iff\quad  a\in P.
  \]
  Thus, the conclusion remains valid even when $R$ has zero-divisors.
\end{Remark}

\begin{Remark}
  The crucial observation is that localization restricts the family of prime ideals above $P$ to those disjoint from $S$. If $R$ is sober, an element
  \[
  a\in\Sect_{\substack{Q\in\Spec(R)\\P\subsetneq Q}} Q\setminus P
  \]
  already belongs to every prime ideal in this smaller family, while remaining outside $P$. Its image $a/1$ therefore witnesses the sobriety condition in the localization.
\end{Remark}

\begin{Corollary}\label{cor:iterated-localization}
  Let $R$ be a sober ring, and let $S,T$ be multiplicatively closed subsets of $R$ not containing zero. Whenever the iterated localization is defined, the ring
  \[
  T^{-1}(S^{-1}R)
  \]
  is sober.
\end{Corollary}

\begin{Remark}
  Theorem~\ref{thm:localization-sobriety} gives a complete answer to the preservation question: there are no additional conditions on $S$ beyond the standard assumptions required for a non-zero localization. In particular, sobriety is preserved under arbitrary localization.
\end{Remark}

\begin{Corollary}\label{cor:localization-at-prime}
  Let $R$ be a sober ring and let $\p$ be a prime ideal of $R$. Then the local ring
  \[
  R_{\p} = (R\setminus\p)^{-1}R
  \]
  is sober.
\end{Corollary}

\begin{Corollary}\label{cor:principal-localization}
  Let $R$ be a sober ring and let $f\in R$ be a non-nilpotent element. Then
  \[
  R_f=R[1/f]
  \]
  is sober.
\end{Corollary}

\subsection{A characterization in terms of the prime spectrum}
We end this section by giving Theorem \ref{thm:localization-sobriety} a different interpretation through the topology of the prime spectrum. For a multiplicative subset $S\subseteq R$, the canonical map
\[
\Spec(S^{-1}R) \To \Spec(R)
\]
identifies $\Spec(S^{-1}R)$ with the subspace
\[
D(S) = \Sect_{s\in S}D(s),
\]
where
\[
D(s) = \{ \p\in\Spec(R)\ :\ s\notin\p \}.
\]
This is the standard spectral description of localization (see \cite{AtiyahMacdonald}).\par\medskip

The prime ideals of the localized ring are therefore precisely the points of $\Spec(R)$ that avoid every element of $S$. The proof of Theorem~\ref{thm:localization-sobriety} shows that the element separating a non-maximal prime from all strictly larger primes in $R$ continues to provide the required separation after restricting to this subspace. Thus, localization does not destroy the sobriety property considered here.

\section{Sobriety of Polynomial Rings}\label{sec:PolynomialRings}
In this section, we investigate additional conditions under which the polynomial ring $R[x]$ is sober whenever the coefficient ring $R$ is sober.

We first explain why sobriety cannot be preserved by polynomial extension without additional assumptions.

\begin{Example}\label{ex:K[x]-not-sober}
  Let $K$ be a field and let $R=K[x]$. It is clear that $K$ is sober since it has dimension zero, but since $R$ is a Dedekind domain with $J(R)=0$, it follows by Proposition \ref{prop:Dedekind} that $R$ is not sober.
\end{Example}

The preceding example shows that the implication
\[
R\text{ sober}\quad \Implies \quad R[x]\text{ sober}
\]
is false in general, hence additional conditions must be required on the spectrum of $R[x]$ in order to guarantee its sobriety.\par\medskip

By the definition of sobriety, the ring $R[x]$ is sober if and only if, for every prime non-maximal prime ideal $\p\in R[x]$, there exists a polynomial $f(x)$ being a sober witness for $\p$. Thus, one has the following criterion.

\begin{Proposition}\label{prop:polynomial-sobriety-criterion}
  Let $R$ be a ring. Then $R[x]$ is sober if and only if, for every prime non-maximal ideal $\p$ of $R[x]$, there exists a polynomial $f(x)\in R[x]$ such that
  \[
  f(x)\notin\p
  \]
  and
  \[
  f(x)\in\q \quad\text{for every prime ideal}\quad \q\supsetneq\p.
  \]
\end{Proposition}

The above result emphasizes that the additional hypothesis must control prime ideals of the polynomial ring, not merely prime ideals of the coefficient ring.

\subsection{A characterization using quotients}
Let $\p$ be a prime non-maximal ideal of $R[x]$. Prime ideals of $R[x]$ containing $\p$ correspond to prime ideals of the domain
\[
A_{\p}=R[x]/\p.
\]
Under this correspondence, the prime ideals strictly containing $\p$ correspond to the non-zero prime ideals of $A_{\p}$. Hence, the following result can be stated.

\begin{Proposition}\label{prop:poly-quotient-intersection}
  Let $R$ be a ring. The polynomial ring $R[x]$ is sober if and only if, for every prime ideal $\p$ of $R[x]$, 
  \[
  \Sect_{\substack{\q\in\Spec(R[x])\\\p\subsetneq\q}} \q/\p \neq (0)
  \]
  in the quotient ring $A_{\p}=R[x]/\p$.
\end{Proposition}
\begin{proof}
  Let $\p$ be a prime non-maximal ideal of $R[x]$. The quotient $A_{\p}=R[x]/\p$ is an integral domain since $\p$ is prime, then image of the intersection of all prime ideals strictly containing $\p$ is exactly the intersection of the corresponding non-zero prime ideals of $A_{\p}$. Thus the sobriety condition
  \[
  \p\neq \Sect_{\substack{\q\in\Spec(R[x])\\\p\subsetneq\q}}\q
  \]
  is equivalent to the existence of a non-zero element in the intersection of all non-zero prime ideals of $A_{\p}$, as stated.
\end{proof}

As a consequence, one can state the following sufficient condition for the polynomial ring $R[x]$ to ensure sobriety.

\begin{Corollary}\label{cor:polynomial-quotient-sufficient}
  Suppose that, for every non-maximal prime ideal $\p$ of $R[x]$, the domain
  \[
  A_{\p}=R[x]/\p
  \]
  has a non-zero element belonging to every non-zero prime ideal. Then $R[x]$ is sober.
\end{Corollary}

\subsection{A finite cofinality condition}
We now establish a sufficient condition for sobriety of $R[x]$. The condition concerns the prime ideals lying strictly above a given non-maximal prime ideal.

\begin{Definition}
  Let $R$ be a ring and let $\p$ be a prime ideal of $R$. A finite family of prime ideals
  \[
  \q_1,\ldots,\q_n
  \]
  is called a \emph{finite cofinal family} above $\p$ if
  \[
  \p\subsetneq\q_i
  \qquad (1\leq i\leq n)
  \]
  and, for every prime ideal $\q$ satisfying $\p\subsetneq\q$, there exists an index $i$ such that
  \[
  \q_i\subseteq\q.
  \]
\end{Definition}

The following theorem gives a convenient sufficient condition.

\begin{Theorem}\label{thm:finite-cofinal-sobriety}
  Let $R$ be a ring. Suppose that, for every non-maximal prime ideal $\p$ of $R[x]$, there exists a finite cofinal family
  \[
  \q_1,\ldots,\q_n
  \]
  of prime ideals strictly containing $\p$. Then $R[x]$ is sober.
\end{Theorem}
\begin{proof}
  Let $\p$ be a non-maximal prime ideal of $R[x]$.
  By hypothesis, there exist prime ideals
  \[
  \p\subsetneq\q_1,\ldots,
  \q_n
  \]
  such that every prime ideal strictly containing $\p$ contains at least one of the $\q_i$. For each $i$, since $\p\subsetneq\q_i$, one can choose an element
  \[
  f_i\in\q_i\setminus\p.
  \]
  We then set
  \[
  f=f_1f_2\cdots f_n.
  \]
  Since $\p$ is prime and none of the $f_i$ belongs to $\p$, we obtain
  \[
  f\notin\p.
  \]
  Let now $\q$ be any prime ideal satisfying such that $\p\subsetneq\q$. By the cofinality assumption, there exists a suitable index $i$ such that
  \[
  \q_i\subseteq\q,
  \]
  then, since $f_i\in\q_i$, we have $f_i\in\q$. Hence,
  \[
  f=f_1\cdots f_n\in\q.
  \]\pagebreak
  
  \noindent
  Therefore,
  \[
  f\in\Sect_{\substack{\q\in\Spec(R[x])\\\p\subsetneq\q}} \q, \qquad f\notin\p.
  \]
  This implies
  \[
  \p\neq \Sect_{\substack{\q\in\Spec(R[x])\\ \p\subsetneq\q}}\q,
  \]
  and since $\p$ was arbitrary, $R[x]$ is sober.
\end{proof}

\begin{Remark}
  Observe that Theorem~\ref{thm:finite-cofinal-sobriety} does not require that $R$ itself be sober. It is a sufficient condition on the prime spectrum of $R[x]$ directly. In particular, it applies when $R$ is sober and the stated finite cofinality condition holds.
\end{Remark}

\subsection{Discussion}
The preceding results in this section provide two complementary perspectives on the sobriety of $R[x]$. The quotient criterion gives an exact characterization in terms of intersections of non-zero prime ideals in the domains $R[x]/\p$, whereas the finite cofinality condition provides a sufficient hypothesis that guarantees sobriety by constructing an explicit witness.

These criteria highlight two aspects of the same phenomenon: the existence of elements lying in every prime ideal strictly above a given non-maximal prime, but not in the prime itself. The finite cofinality condition ensures the existence of such elements through a finite product, while the quotient criterion expresses their existence directly in terms of prime intersections.

\section{Conclusion}
In this paper, we have investigated the behavior of sober rings under several standard ring constructions. We have established basic structural results and obtained a characterization of sobriety for Dedekind domains that are not fields. We have also shown that sobriety is preserved under quotients and localizations.

Polynomial extensions exhibit a different behavior: even when $R$ is sober, $R[x]$ need not be sober. We have therefore considered conditions ensuring the sobriety of $R[x]$.

These results highlight the role played by the intersection behavior of prime ideals in the study of sober rings. We are confident that they will provide a useful starting point for further investigations into sobriety and its interaction with other ring constructions and classes.

\subsection*{Acknowledgment}
The authors thank Prof. Marilena Crupi for helpful suggestions. E.~Lax acknowledges support from the GNSAGA group of INdAM (Italy).

\end{document}